\documentclass[12pt]{amsart}
\usepackage {enumerate}
\usepackage{setspace}
\usepackage{caption}
\usepackage{mathrsfs}
\usepackage{hyperref}
\usepackage{esint}
\usepackage{amssymb}
\usepackage{graphicx}
\usepackage{epstopdf}
\usepackage{amsthm}
\usepackage{appendix}
\usepackage{color}
\usepackage{lipsum}

\newtheorem{theorem}{Theorem}[section]
\newtheorem{lemma}[theorem]{Lemma}
\newtheorem{corollary}[theorem]{Corollary}
\newtheorem{proposition}[theorem]{Proposition}

\numberwithin{equation}{section}

\theoremstyle {definition}
\newtheorem{definition}{Definition}[section]
\newtheorem{remark}{Remark}[section]

\makeatletter
\@namedef{subjclassname@2020}{%
  \textup{2020} Mathematics Subject Classification}
\makeatother

\DeclareMathOperator{\Div}{div}

\DeclareMathOperator{\R}{\mathbb R}
\DeclareMathOperator{\N}{\mathbb N}

\newcommand{\diam}{\mathrm{diam}}
\DeclareMathOperator{\E}{\mathcal{E}}

\newcommand{\sym}{\ensuremath{\mathrm{sym}}}
\newcommand{\curl}{\mathrm{curl}}

\begin{document}
\title[2D Lagrangian mean curvature equation]{Ill-posedness of the Dirichlet problem for 2D Lagrangian mean curvature equations}

\begin{abstract}
We investigate the Dirichlet problem of the two-dimensional Lagrangian mean curvature equation in bounded domains. Infinitely many $C^{1, \alpha} (\alpha\in (0,\frac{1}{5}))$ very weak solutions are built through Nash-Kuiper construction. Moreover, we note there are infinitely many $C^{1, \alpha}$ very weak solutions that can not be improved to be $C^{2, \alpha}$.

\end{abstract}

\author{Wentao Cao}
\address{Academy for Multidisciplinary Studies, Capital Normal University, Beijing, 100048, China}
    \email{cwtmath@cnu.edu.cn}
\author{Zhehui Wang}
\address{School of Sciences, Great Bay University, Dongguan, 523000, China}
\email{wangzhehui@gbu.edu.cn}

\maketitle
\section{Introduction}
In the present paper, we study the Dirichlet problem for the two-dimensional Lagrangian mean curvature equation
\begin{equation}\label{eq-dirich-lag}
    \begin{cases}
        \cos\Theta\Delta v+\sin\Theta\left(\det\nabla^2 v-1\right)=0 \text{ in }\Omega,\\
        v=g \text{ on }\partial\Omega,
    \end{cases}
\end{equation}
where $\Omega\subset\R^2$ is a bounded domain, and the phase function $\Theta: \overline{\Omega}\to (-\pi, \pi)$ is a given $C^{2,\kappa}(\kappa\in(0,1))$ function. The equation
\begin{equation}\label{lag-eq2}
\cos\Theta\Delta v+\sin\Theta\left(\det\nabla^2 v-1\right)=0
\end{equation}
has an equivalent form
\begin{equation}\label{lag-eq1-1}
\arctan \lambda_1+\arctan\lambda_2=\Theta
\end{equation}
with $\lambda_1, \lambda_2$  being eigenvalues of $\nabla^2 v$. The phase function $\Theta$
is a potential for the mean curvature of the Lagrangian submanifold
  $$\Sigma:=\{(x,\nabla v(x)); x\in\Omega\}$$
or called the gradient graph.

Generally, the $n$-dimensional Lagrangian mean curvature equation is 
\begin{equation}\label{lag-eq1}
\sum_{k=1}^{n}\arctan \lambda_k=\Theta,
\end{equation}
where $\lambda_i$'s are eigenvalues of $\nabla^2 v$. The mean curvature vector of the $n$-dimensional gradient graph $\Sigma$ is equal to $J\nabla_g\Theta$, where $g=\mathrm{Id}+(\nabla^2 v)^2$ is the induced metric on $\Sigma$, and $J(x, y)=(-y, x)$ is a rotation in $\R^n\times\R^n$. 
In particular, when $\Theta$ is a constant, \eqref{lag-eq1} is called the special Lagrangian equation, which originates in the special Lagrangian geometry introduced by Harvey-Lawson \cite{HL}. Note $\Sigma$ is a minimal surface in $\R^n\times\R^n$ if and only if $\Theta$ is a constant. 

\subsection{Background}
For the $n$-dimensional equation \eqref{lag-eq1}, Yuan \cite{Yuan06} noticed that the level set $\{\lambda:=(\lambda_1,\cdots, \lambda_n)\in\R^n| \lambda \text{ satisfies equation } \eqref{lag-eq1}\}$ is convex only when $|\Theta|\geq\frac{(n-2)\pi}{2}$. Thus the phase $\Theta$ is called critical if $|\Theta|=\frac{(n-2)\pi}{2}$, and called supercritical (or subcritical) if $|\Theta|>\frac{(n-2)\pi}{2}$ (or $|\Theta|<\frac{(n-2)\pi}{2}$). When $n=2$, viewing \eqref{lag-eq2} as Monge-Amp\`{e}re type equations, Heinz established some regularity results in \cite{Heinz}. For higher dimensions, a priori estimates for solutions of the special Lagrangian equation with critical and supercritical phases were obtained by Warren-Yuan \cite{WY09, WY10}, Wang-Yuan \cite{WaY}, Chen-Warren-Yuan \cite{CWY}, respectively.  Moreover, the Hessian estimate for Lagrangian mean curvature equation \eqref{lag-eq1} with $C^{1,1}$ critical and supercritical phases was obtained by Bhattacharya \cite{Bha21, Bha22}, while the gradient estimate for $C^2$ critical and supercritical phases was achieved by Bhattacharya-Shanker-Mooney \cite{BSM}. Hence, combining a priori estimates in \cite{Bha21, Bha22, BSM}, they solved the Dirichlet problem for Lagrangian mean curvature equation \eqref{lag-eq1} in uniformly convex bounded domains with the phase function satisfying $\frac{(n-2)\pi}{2}\leq \Theta<\frac{n\pi}{2}$. Classical solvability of the Dirichlet problem for the Lagrangian mean curvature equation \eqref{lag-eq1} was also established by Lu \cite{Lu2023}, and Collins, Picard and Wu \cite{CPW2017}. Besides, Harvey and Lawson \cite{HL2, HL3}, and Cirant and Payne \cite{CP2021} proved the existence of the viscosity solution to the Lagrangian mean curvature equation \eqref{lag-eq1}. However, interior estimates fail for the subcritical case due to the examples constructed in Nadirashvili-Vl\u{a}du\c{t} \cite{NV}, Wang-Yuan \cite{WaY2}, and Mooney-Savin \cite{MS}.

In this paper, we consider not only classical solutions  but also very weak solutions to \eqref{eq-dirich-lag}. In fact, the very weak solution to the Monge-Amp\`ere equation was first introduced by Lewicka-Pakzad \cite{LePa} via very weak Hessian defined by Iwaniec \cite{Iwan}. The two-dimensional very weak Hessian coincides with the Monge-Ampère measure for any convex solution (\cite{DGG}). Moreover, Without boundary conditions, Lewicka-Pakzad \cite{LePa}, Sz\'{e}kelyhidi  and the first author \cite{CSze}, and Hirsch-Inauen and the first author \cite{CHI} proved that there are infinitely many very weak solutions to the two dimensional Monge-Amp\`{e}re equation of different regularity. After imposing boundary data, the first author in \cite{Cao} studied very weak solutions to the Monge-Amp\`ere equation. Recently, Li-Qiu \cite{LQ} also constructed infinitely many very weak solutions of the Dirichlet problem for 2-Hessian equations.
All the above approaches to constructing very weak solutions rely on the
Nash-Kuiper construction, which was originated from isometric embeddings (the Nash-Kuiper theorem; see \cite{Nash54, Kui55, CDS12, CS19-a, DIS, Sze}), and developed to be convex integration to prove the $h$-principle (see \cite{Gro}). In recent decades, the technique of convex integration becomes a powerful approach to showing non-uniqueness of nonlinear partial differential equations.

\subsection{Main result and strategies}
Along the line of very weak solutions to the Monge-Amp\`ere equation, we also introduce the concept of very weak solutions to the Dirichlet problem \eqref{eq-dirich-lag}.

\begin{definition}[Very weak solution]\label{defvweak}
$v\in C^1(\Omega)$ is a very weak solution to the Dirichlet problem \eqref{eq-dirich-lag} in $\Omega$, if $v=g$ on $\partial\Omega$, and it satisfies the following equation in the very weak sense in $\Omega$:
\begin{equation}\label{lag-eq3}
  \cos\Theta \mbox{ curl curl } (v\cdot \mbox{Id})+\sin\Theta\left(-\frac{1}{2}\mbox{ curl curl } (\nabla v\otimes\nabla v)-1\right)=0;
  \end{equation}
namely, for any $\phi\in C^\infty_0(\Omega)$, 
\begin{align}
 &\int_\Omega\partial_1v\partial_2v\partial_{12}^2(\phi\sin\Theta)-\frac{1}{2}(\partial_1v)^2\partial_{22}^2(\phi\sin\Theta)-\frac{1}{2}(\partial_2v)^2\partial_{11}^2(\phi\sin\Theta) dx \nonumber\\
 &\quad +\int_\Omega v\Delta (\phi \cos\Theta)dx-\int_\Omega \phi\sin\Theta dx=0.
\end{align}
\end{definition}
The operator $\mathrm{curl}$ in two-dimension is defined by 
    \begin{align*}
    \mbox{ curl } \xi&=\partial_2\xi_1-\partial_1\xi_2 \mbox{ for } \xi=(\xi_1, \xi_2)\in\mathbb{R}^2,\\
    \mbox{ curl }A&=(\partial_2A_{11}-\partial_1A_{12}, \partial_2A_{21}-\partial_1A_{22}) \mbox{ for } A=(A_{ij})_{2\times 2}\in \R^{2\times 2}.
    \end{align*}
In particular, $\mbox{ curl curl } (v\cdot \mbox{Id})$ is understood as $\Delta v$ in the very weak sense. 

Now we are ready to state our main result.
\begin{theorem}\label{infity-sol}
Let $\Omega\subset\R^2$ be a bounded simply connected domain  with a $C^{2,\kappa}$ boundary $\partial\Omega$ for any given $\kappa\in(0, 1)$. Suppose $g\in C^{2,\kappa}(\partial\Omega)$ and the phase $\Theta:\overline{\Omega}\to (-\pi, \pi)$ is a $C^{2,\kappa}$ function. Then the following statements hold.
\begin{itemize}
    \item[(1)] There exists some small positive number $c_1$ such that if $|\sin\Theta|\leq c_1$, then the Dirichlet problem \eqref{eq-dirich-lag} admits a unique $C^{2, \kappa}(\Omega)$ classical solution.
    \item[(2)] Given any constant $c_2\in(0, 1]$, the Dirichlet problem \eqref{eq-dirich-lag} admits infinitely many $C^{1, \alpha}(\Omega)$ very weak solutions for any $\alpha\in(0,\frac15)$ whenever $|\sin\Theta|\geq c_2$.
\end{itemize}
\end{theorem}

\begin{remark}
    The assumption of simply connectedness is only needed for applying Lemma \ref{diag} to the proof of Proposition \ref{p:stage}.
\end{remark}
Noting $c_2$ is arbitrary, as a quick observation, we have
\begin{corollary}
For any $c_3\in(0, c_1)$, if $c_3\leq|\sin\Theta|\leq c_1$, then there are infinitely many $C^{1, \alpha}$ very weak solutions to the Dirichlet problem \eqref{eq-dirich-lag} which can not be improved to be $C^{2,\alpha}$.
\end{corollary}
 Our strategy to show Theorem \ref{infity-sol} (1) is considering \eqref{lag-eq2} as a perturbation of a Laplacian equation and then applying the implicit function theorem. 
 As for Theorem \ref{infity-sol} (2), with $|\sin\Theta|\geq c_2>0$, the problem is reduced to solve
 $$\cot\Theta\Delta v+\det\nabla^2 v-1=0 \text{ in the very weak sense}.$$
 We will solve the above equation by the Nash-Kuiper construction. The main obstacle is how to treat terms with the phase function, compared with \cite{Cao}, which will be regarded as a Laplacian of some function, i.e.,
 $$\cot\Theta\Delta v=\Delta(v\cot\Theta-V), \text{ in the very weak sense,}$$
 with $$\Delta V=2\nabla v\cdot\nabla\cot\Theta+v\Delta\cot\Theta.$$
 Hence, the problem is turned to solve $$\curl\,\curl\,\big(\mbox{ sym }\nabla w+\frac{1}{2}\nabla v\otimes \nabla v-(v\cot\Theta)\cdot\mathrm{Id}+V\cdot\mathrm{Id}\big)=-1,$$ in the very weak sense, and then $V\cdot\textrm{Id}$ becomes an error term of a lower order. With the elliptic estimate in Lemma \ref{elliptic}, we are able to control it and arrive at our conclusion.

We remark that the method proving Theorem \ref{infity-sol} (2) has more general applications.
\begin{remark}\label{genrel}
Let
\begin{equation*}
F_\tau(\lambda)=\left\{
\begin{aligned}
    &\frac{1}{2}\big(\ln \lambda_1+\ln \lambda_2\big), \qquad \tau=0,\\
    &\frac{\sqrt{a^2+1}}{2b}\big(\ln\frac{\lambda_1+a-b}{\lambda_1+a+b}+\ln\frac{\lambda_2+a-b}{\lambda_2+a+b}\big), \qquad 0<\tau<\frac{\pi}{4},\\
    &-\sqrt{2}\big(\frac{1}{1+\lambda_1}+\frac{1}{1+\lambda_2}\big), \qquad \tau=\frac{\pi}{4},\\
    &\frac{\sqrt{a^2+1}}{b}\big(\arctan\frac{\lambda_1+a-b}{\lambda_1+a+b}+\arctan\frac{\lambda_2+a-b}{\lambda_2+a+b}\big), \qquad \frac{\pi}{4}<\tau<\frac{\pi}{2},\\
    &\arctan\lambda_1+\arctan\lambda_2, \qquad \tau=\frac{\pi}{2},
\end{aligned}\right.
\end{equation*}
where $a=\cot\tau$, and $b=\sqrt{|\cot^2\tau-1|}$. With slight modifications, we can establish the result like Theorem \ref{infity-sol} (2) for the following Dirichlet problem
\begin{align}\label{general}
\begin{cases}
F_\tau(\lambda)=\Theta, &\mbox{ in } \Omega,\\
\qquad v=g, &\mbox{ on } \partial\Omega,
\end{cases}
\end{align}
where $\lambda_1$, $\lambda_2$ are eigenvalues of $\nabla^2 v$. We note the equation $F_\tau(\lambda)=\Theta$ can be formulated as the equation of the following form
$$P\Delta u+Q\det\nabla^2 u=R.$$

Equations in \eqref{general} are originated from the study of the gradient graph $\Sigma$. By Warren \cite{Wa}, if $\Sigma$ is a minimal Lagrangian graph in $(\R^2\times\R^2, g_\tau)$ with
$$g_\tau=\sin\tau \big(\sum_{i=1}^2d x_i\otimes d x_i+\sum_{j=1}^2d y_j\otimes d y_j\big)+\cos\tau\big(\sum_{i=1}^2 d x_i\otimes d y_i+\sum_{j=1}^2 d y_j\otimes d x_j\big),$$ 
then $$F_\tau(\lambda(\nabla^2 v))=\text{constant}.$$
We refer readers to the results about gradient graphs in general dimensions in \cite{Wa}.
\end{remark}

\subsection{Organization and notations}
The rest of the paper is organized as follows. Theorem \ref{infity-sol} (1) will be proved in Section \ref{s-proof-1}. Our stage proposition, which is our main contribution, and  the proof for  Theorem \ref{infity-sol} (2)  will be presented in Section \ref{s-stage} and Section \ref{pfmain} respectively.

Finally, we provide some notations for convenience.  For any function $h: \Omega\to\R$, set 
\begin{align*}
&\|h\|_0=\sup_{\Omega}|h|,\,\,\|h\|_m=\sum_{j=0}^m\max_{|\beta|=j}\|\partial^\beta h\|_0,\\
& \|h\|_{m+\alpha}=\|h\|_m+[h]_{m+\alpha},
\end{align*}
where the H\"{o}lder semi-norm is defined by 
$$[h]_{m+\alpha}=\max_{|\beta|=m}\sup_{x\neq y}\frac{|\partial^\beta h(x)-\partial^\beta h(y)|}{|x-y|^\alpha}.$$
Here $m$ is a nonnegative integer, $\alpha\in (0, 1]$  a constant, and $\beta$ be a multi-index. We denote by $B_r$ the ball in $\R^2$ with radius $r$ centering at $0$. Let $\varphi\in C^\infty_0(B_1)$ be a standard two-dimensional mollifier. We define $\varphi_l(x)=l^{-2}\varphi(l^{-1}x)$, and $l>0$ is called the length-scale. We also define
$$\tilde{h}(x)=h* \varphi_l(x)=l^{-2}\int_{\R^2}h(x-y)\varphi(\frac{y}{l}) dy.$$
Some properties of mollification can be found in Lemma \ref{holder}. 

Besides, for $\xi=(\xi_1, \xi_2)\in\mathbb{R}^2$ and $\eta=(\eta_1, \eta_2)\in\mathbb{R}^2$, we write $\xi\otimes\eta$ as a $2\times 2$-matrix with entries $(\xi\otimes\eta)_{ij}=\xi_i\eta_j$. 
By a direct calculation, for a $C^3$-function $w$, we have $$\det \nabla^2 w=-\frac{1}{2}\mbox{curl curl }(\nabla w\otimes\nabla w)
    \mbox{ and } \Delta w=\mbox{curl curl } (w\cdot\mathrm{Id}).$$
For a $C^1$-function $w: \Omega\to \R^2$, note $\nabla w=(\partial_iw_j)_{2\times 2}$ and $$\sym \nabla w=\frac{1}{2}(\nabla w+\nabla w^T).$$

\section*{Acknowledgements}
Wentao Cao's research was supported in part by National Natural Science Foundation of China (Grant No.\,12471224). Zhehui Wang's research was supported in part by Young Scientists Fund of National Natural Science Foundation of China (Grant No.\,12401247), and Guangdong Basic and Applied Basic Research Foundation (Grant No.\,2023A1515110910).  

\section{Proof of Theorem \ref{infity-sol} (1)}\label{s-proof-1}
The proof of Theorem \ref{infity-sol} (1) is based on the application of the implicit function theorem.

\begin{proof}
Let $\widetilde{g}\in C^{2,\kappa}(\overline{\Omega})$ such that $\widetilde{g}=g$ on $\partial\Omega$. Set 
$$\mathcal{X}=\{\psi\in C^{2, \kappa}(\overline{\Omega}): \psi=0 \text{ on }\partial \Omega\}.$$ 
Define 
$G: \mathcal{X}\times C^\kappa(\overline\Omega)\rightarrow C^\kappa(\overline{\Omega})$ by
$$G(\psi, \Psi)=\Delta \psi+\Psi\cdot(\det\nabla^2 (\psi+\widetilde{g})-1)+\Delta\widetilde{g}.$$ Then we have the following conclusion.

\noindent{\bf Claim}: 
    There are a unique $\psi\in\mathcal{X}$ and a small enough constant $\mu\in(0,1)$ such that $G(\psi,\Psi)=0$
    for any $\Psi$ satisfying $\|\Psi\|_\kappa\leq\mu$.
    
In fact, the Fr\'{e}chet derivative of $G$ with respect to $\psi$ is given by
\begin{align*}
G_\psi(\psi, \Psi)\vartheta=\Delta \vartheta+\frac{d\,}{dt}\Big |_{t=0}\Psi\cdot(\det\nabla^2 (\psi+\widetilde{g}+t\vartheta)-1),
\end{align*}
and hence $G_\psi(\psi, 0)\vartheta=\Delta\vartheta$ for any $\vartheta\in \mathcal{X}$. 
Therefore, for any $\psi\in \mathcal X$,  $G_\psi(\psi, 0)$ is an invertible operator from $\mathcal{X}$ to $C^\kappa(\overline{\Omega})$. 

Next, take $\psi_0\in \mathcal X$ such that $\Delta\psi_0=-\Delta\widetilde{g}$. 
The existence of such $\psi_0$ is provided by the Schauder theory. 
By the implicit function theorem, the equation $G(\psi, \Psi)=0$ admits a unique solution $\psi=\psi_\Psi\in \mathcal X$ 
for each $\Psi$ with $\|\Psi\|_{C^\kappa(\overline\Omega)}\leq\mu$, where $\mu$ is a sufficiently small constant. Moreover, $\psi_\Psi$ is close to $\psi_0$ in the $C^{2,\kappa}(\overline{\Omega})$-norm. 

Let $\Theta: \overline{\Omega}\to (-\pi, \pi)$ be a $C^{1,\kappa}$ function. By the interpolation, we note 
$$\|f\|_{C^\kappa(B_r)}\leq \sigma\|f\|_{C^{1, \kappa}(B_r)}+(\frac{C}{\sigma^\kappa}+1)\|f\|_{L^\infty(B_r)} $$ 
for any $0<\sigma\leq r$ and $f\in C^{1, \kappa}(B_r)$. Hence, if $|\sin\Theta|\leq c_1$ with the small constant $c_1$ to be fixed, then we can get $$\|\sin\Theta\|_\kappa\leq C(\kappa, \|\Theta\|_{1+\kappa})(c_1+c_1^{1-\kappa}),$$ 
where $C(\kappa, \|\Theta\|_{1+\kappa})$ is a constant depending only on $ \kappa, \|\Theta\|_{1+\kappa}.$
By a direct computation,
\begin{align*}
\|\tan\Theta\|_\kappa&\leq \|\sin\Theta\|_0\|\frac{1}{\cos\Theta}\|_\kappa+\|\sin\Theta\|_\kappa\|\frac{1}{\cos\Theta}\|_0\\
&\leq 4\|\sin\Theta\|_\kappa\|\frac{1}{\cos\Theta}\|_1\\
&\leq C(\kappa, \|\Theta\|_{1+\kappa})\|\sin\Theta\|_\kappa (\|\frac{1}{\cos\Theta}\|_0+\|\frac{1}{\cos^2\Theta}\|_0\|\sin\Theta\|_0)\\
&\leq C(\kappa, \|\Theta\|_{1+\kappa})(c_1+c_1^{1-\kappa})(\frac{1}{\sqrt{1-c_1^2}}+\frac{c_1}{1-c_1^2}).  
\end{align*}
Now, we choose $c_1$ small enough such that $\|\tan\Theta\|_{C^\kappa(\overline\Omega)}\leq\mu$.
By the above Claim, there is a unique $\psi\in \mathcal X$ such that 
$$
\Delta \psi+\tan\Theta(\det\nabla^2 (\psi+\widetilde{g})-1)+\Delta\widetilde{g}=0.
$$
Then, $v:=\psi+\widetilde{g}$ uniquely satisfies
\begin{equation*}
\left\{
\begin{aligned}
&\Delta v+\tan\Theta(\det\nabla^2 v-1)=0 \mbox{ in } \Omega,\\
&v=g \mbox{ on } \partial\Omega.
\end{aligned}
    \right.
\end{equation*}

Hence, we complete the proof of Theorem \ref{infity-sol} (1).
\end{proof}

\section{The Stage Proposition}\label{s-stage}
The rest of the paper is devoted to showing Theorem \ref{infity-sol} (2). Similar to the treatment in \cite{Cao,CSze,LePa} for the Monge-Amp\`ere equation, we split \eqref{lag-eq2} (after dividing $\sin\Theta$) into  a differential system
\begin{align}
& -\mbox{curl curl } A=1,\label{eq-split-1}\\
&A=\mbox{ sym }\nabla w+\frac12\nabla v\otimes\nabla v-(v\cot\Theta)\cdot\mathrm{Id}+V\cdot\mathrm{Id},\label{eq-split-2}
\end{align}
where $V$ is the solution to
\begin{equation*}
\left\{
\begin{aligned}
        \Delta V &=F(v), \mbox{ in } \Omega,\\
        V&=0, \mbox{ on } \partial\Omega,
\end{aligned}
\right.
\end{equation*}
and $F(v)$ is defined as
\begin{equation}\label{Fv}
F(v)=2\nabla v\cdot\nabla \cot\Theta+v\Delta\cot\Theta.
\end{equation}
In fact, by a direct calculation, we note that $v$ satisfying \eqref{eq-split-1}-\eqref{eq-split-2} is a solution of the equation \eqref{lag-eq3} in the very weak sense.

To solve \eqref{eq-split-2}, our strategy is applying the technique of Nash-Kuiper construction to construct a sequence of subsolutions satisfying the boundary condition, whose limit is our desired solution.  
Set
\begin{equation}\label{Dvw}
D(v, w):=\mbox{ sym }\nabla w+\frac{1}{2}\nabla v\otimes \nabla v-(v\cot\Theta)\cdot\mathrm{Id}+V\cdot\mathrm{Id},
\end{equation}
then such subsolutions, which are called adapted subsolutions, are defined as below.
\begin{definition}\label{d-sub}
A pair of mappings $(v, w):\overline{\Omega}\to\R\times\R^{2}$ is called a $C^{1, \beta}(\beta<1)$ adapted subsolution of \eqref{eq-split-2} if $(v, w)\in C^{1, \beta}$ and
$$
A-D(v, w)=\rho^2(\textrm{Id}+H) \mbox{ in } \overline{\Omega},
$$
where $D(v, w)$ is as \eqref{Dvw}, and $\rho\in C(\overline{\Omega})$ is a nonnegative function with $\partial\Omega\subset\{x:\rho(x)=0\}$, and $H\in C(\overline{\Omega};\R^{2\times 2}_{sym})$ is a symmetric matrix such that $\textrm{Id}+H$ is positive definite. 
\end{definition}
Note that  the definition is a slight modification of that in \cite{Cao}.
To construct new adapted subsolutions, we need to add the deficit $A-D(v, w)$ to the given adapted subsolution $v.$ Since the deficit is of the form $\rho^2(\mathrm{Id}+H)$, we require the following stage proposition (in the terminology of Nash  \cite{Nash54}) for iteration, which is parallel to Proposition 3.1 in \cite{Cao} and whose proof is different due to the new error terms. The new error terms are then controlled through Lemma \ref{elliptic}.

\begin{proposition}\label{p:stage}
    Suppose $\Sigma\subset \R^2$ is a simply connected open bounded set with a $C^2$ boundary. Let $v\in C^2(\overline\Sigma)$ and $w\in C^2(\Sigma, \R^2)$ such that
\begin{align}
&\|v\|_1+\|w\|_1\leq \gamma,\label{eq-vw-c1}\\
&\|v\|_2+\|w\|_2\leq \delta^{\frac{1}{2}}\lambda,\label{eq-vw-c2}
\end{align}
and let $\rho\in C^1(\overline\Sigma)$, $H\in C^1(\overline\Sigma, \R^{2\times 2}_{sym})$ be such that 
\begin{align}
&\|\rho\|_0\leq \delta^\frac12,\quad\|\nabla \rho\|_0\leq \delta^{\frac{1}{2}}\lambda,\label{rho01}\\
&\|H\|_0\leq \lambda^{-\alpha},\quad\|\nabla H\|_0\leq\lambda^{1-\alpha}, \label{H01}
\end{align}
for some $\gamma>1$, $0<\delta,\,\alpha<1$ and $\lambda>1$ with $2\leq \lambda^\alpha$ and $\delta^{1/2}\lambda>1$. Then for any $\tau>1$, there exist functions $v^\star\in C^2(\overline\Sigma)$, $w^\star\in C^2(\overline\Sigma, \R^2)$ and a symmetric matrix $\E\in C^1(\overline\Sigma, \R^{2\times2}_{sym})$ such that 
\begin{align}
    & D(v^\star, w^\star)=D(v, w)+\rho^2(\mathrm{Id}+H)+\E,\,\, \mbox{in } \overline\Sigma,\\
    & (v^\star, w^\star)=(v, w),\,\, \E=0,\,\, \mbox{on } \overline\Sigma\setminus(\mathrm{supp}\rho+B_{\lambda^{-\tau}}), \label{eq-support}
\end{align}
and they satisfy the following estimates
\begin{align}
\|v^\star-v\|_0+\|w^\star-w\|_0&\leq C\delta^\frac12\lambda^{-\tau},\label{eq-vw-diff-0}\\
\|v^\star-v\|_1+\|w^\star-w\|_1&\leq C\delta^\frac{1}{2},\label{eq-vw-diff-1}\\
\|v^\star\|_2+\|w^\star\|_2&\leq C\delta^\frac{1}{2}\lambda^{2\tau-1}, \label{eq-vw-diff-2}
\end{align}
and
\begin{equation}
\|\E\|_0\leq \bar C\delta\lambda^{1-\tau},\quad \|\E\|_1\leq \bar C\delta\lambda^\tau, \label{eq-deficit-c01}
\end{equation}
where constant $C$ and $\bar C$ depend only on $\gamma$ and $\alpha$.
\end{proposition}

Our proof for the proposition above differs from that for Proposition 3.1 in \cite{Cao} due to the presence of the phase function $\Theta$. Indeed, the key ingredients of our proof consist of applying the elliptic estimates in Lemma \ref{elliptic} to control the error terms arising from $\Theta$.

\begin{proof} We divide the proof into three steps. 

{\it Step 1. Mollification.} Take the lengthscale $\ell=\lambda^{-\tau}$. Set $\tilde\rho=\rho*\varphi_{\ell}$ and $\tilde{H}=H*\varphi_{\ell}$. Then by \eqref{rho01}, \eqref{H01} and Lemma \ref{holder}, we know,
\begin{align*}
    &\|\tilde\rho\|_0\leq C\delta^{\frac{1}{2}},\quad \|\tilde\rho\|_j\leq C(j)\ell^{1-j}[\rho]_1\leq C(j)\delta^{\frac{1}{2}}\lambda \ell^{1-j}, \\
    &\|\tilde \rho-\rho\|_0\leq C\ell[\rho]_1\leq C\delta^{\frac{1}{2}}\lambda \ell,\\
    &\|\tilde H\|_0\leq C\lambda^{-\alpha}, \quad \|\tilde H\|_j\leq C(j)\ell^{1-j}[H]_1\leq C(j)\lambda^{1-\alpha} \ell^{1-j},\\
    &\|\tilde H-H\|_0\leq C\ell[H]_1\leq C\lambda^{1-\alpha} \ell,
\end{align*}
for $j\in \N_+$. Let $\tilde{h}=\tilde{\rho}^2(\mathrm{Id}+\tilde{H})$ and $h=\rho^2(\mathrm{Id}+{H})$, then
\begin{align*}
    \|\tilde{h}-h\|_0&\leq\|\tilde{\rho}^2-\rho^2\|_0+\|\tilde{\rho}^2\tilde H-\rho^2H\|_0\leq C\delta\lambda \ell+C\delta\lambda^{1-\alpha}\ell\leq C\delta\lambda^{1-\tau}, \\
    \|\tilde{h}-h\|_1
    &\leq\|\tilde{\rho}^2-\rho^2\|_1+\|\tilde{\rho}^2(\tilde H-H)\|_1+\|(\tilde{\rho}^2-\rho^2)H\|_1\\
    &\leq C(\delta\lambda+\delta\lambda^{1-\alpha})\leq C\delta\lambda. 
\end{align*}
Besides, it holds that
$$\textrm{Id}+\tilde H\geq(1-\lambda^{-\alpha})\textrm{Id}>\tfrac12\textrm{Id},$$
and
$$\|\textrm{Id}+\tilde H\|_{\alpha'}\leq 1+\|\tilde H\|_{\alpha'}\leq C(1+\lambda^{\alpha'-\alpha})\leq C,$$
provided $0<\alpha'<\alpha.$
Thus applying Lemma \ref{diag} to $\textrm{Id}+\tilde H$, we get
$$\textrm{Id}+\tilde H=\tilde a^2(\nabla\Phi_1\otimes\nabla\Phi_1+\nabla\Phi_2\otimes\nabla\Phi_2),$$
with $\tilde a, \Phi=(\Phi_1, \Phi_2)$ satisfying
\begin{equation*}
\begin{split}
&\tilde a\geq C^{-1},\quad \det(\nabla\Phi)\geq C^{-1},\\
&\|\tilde a\|_{j+\alpha'}+\|\nabla\Phi\|_{j+\alpha'}
\leq\|\mathrm{Id}+\tilde{H}\|_{j+\alpha'}\\
\leq &C(1+\|\tilde{H}\|_j\ell^{-\alpha'})
\leq C(j,\alpha', \alpha)\lambda^{1-\alpha}\ell^{1-j-\alpha'},
\end{split}
\end{equation*}
for any $j\geq 1$, and 
\begin{equation*}
\|\tilde a\|_{\alpha'}+\|\nabla\Phi\|_{\alpha'}\leq \|\mathrm{Id}+\tilde{H}\|_{\alpha'}\leq C,
\end{equation*}
Taking $\alpha'\leq\frac\alpha\tau,$ we have
\begin{equation}\label{eq-phi-cj}
C^{-1}\leq |\nabla\Phi_1|,\,|\nabla\Phi_2|\leq C,\quad \|\nabla\Phi\|_j\leq C(j)\lambda\ell^{1-j}\leq C(j)\lambda^{1-\tau+\tau j},
\end{equation}
and
$$\|\tilde a\|_j\leq C(j)\lambda^{1-\tau+\tau j} \text{ for any } j\geq1. $$
Setting $a=\tilde a\tilde \rho$, we obtain
$$\tilde h=a^2(\nabla\Phi_1\otimes\nabla\Phi_1+\nabla\Phi_2\otimes\nabla\Phi_2),$$
and
\begin{equation}\label{eq-a-c2}
\|a\|_0\leq C\delta^{1/2},\quad \|a\|_j\leq C(j)\delta^{1/2}\lambda^{1-\tau+\tau j}
\end{equation}
for any $j\geq1.$

\smallskip

{\it Step 2. Adding the first rank-one matrix.} Set $\tilde v=v*\varphi_{\ell}$ and $\tilde{w}=w*\varphi_{\ell}$.  With Lemma \ref{holder}, we get
\begin{align}
    & \|\tilde v-v\|_1+\|\tilde w-w\|_1\leq C\ell(\|v\|_2+\|w\|_2)\leq C\delta^{\frac{1}{2}}\lambda^{1-\tau},\label{e:tilde-vw-C1}\\
    &\|\tilde v\|_{2+j}+\|\tilde w\|_{2+j}\leq C\ell^{-j}(\|v\|_2+\|w\|_2)\leq C\delta^{\frac{1}{2}}\lambda^{1+\tau j}.\label{e:tilde-vw-C2}
\end{align}
Define
\begin{align*}
v^\lozenge(x)=&v(x)+\lambda^{-\tau}\Gamma_1(a(x), \lambda^{\tau}\Phi_1(x)),\\
w^\lozenge(x)=&w(x)-\lambda^{-\tau}\Gamma_1(a(x), \lambda^{\tau}\Phi_1(x))\nabla\tilde{v}(x)\\
&+\lambda^{-\tau}\Gamma_2(a(x), \lambda^{\tau}\Phi_1(x))\nabla\Phi_1(x),
\end{align*}
where  $\Gamma_i, i=1, 2$ are defined as follows:
\begin{equation}\label{Gamma}
\Gamma_1(s, t)=\frac{s}{\pi}\sin(2\pi t),\quad \Gamma_2(s, t)=-\frac{s^2}{4\pi}\sin(4\pi t).
\end{equation} 
It is easy to see that $\Gamma_i(s, t)=\Gamma_i(s, t+1)$, $i=1, 2$ and 
  \begin{equation}\label{e-gamma12}
    \frac{1}{2}|\partial_t\Gamma_1(s, t)|^2+\partial_t\Gamma_2(s, t)=s^2.
  \end{equation}
From both definitions of $\Gamma_i, i=1, 2$ we gain
\begin{equation}\label{eq-vw-natural-1}
(v^\lozenge, w^\lozenge)=(v, w) \textit{ on } \overline\Omega\setminus\textrm{supp } a.
\end{equation}
Observe that
\begin{equation}\label{e:vj}
\|v^\lozenge-v\|_j\leq \lambda^{-\tau}\|\Gamma_1\|_j,
\end{equation}
and
\begin{equation}\label{e:wj}
\begin{split}
\|w^\lozenge-w\|_j\leq C(j)\lambda^{-\tau}(&\|\Gamma_1\|_j\|\tilde{v}\|_1+\|\Gamma_1\|_0\|\tilde{v}\|_{j+1}\\
&+\|\Gamma_2\|_j\|\nabla\Phi_1\|_0+\|\Gamma_2\|_0\|\nabla\Phi_1\|_j),
\end{split}
\end{equation}
for $ j=0, 1, 2.$  Here $\|\Gamma_i\|_j$ denotes the $C^j$ norms of the function $x\rightarrow\Gamma_i(a(x), \lambda^\tau\Phi_1(x)).$  
\begin{lemma}\label{Gamma-estimate} 
Assume all the assumptions of Proposition \ref{p:stage} are fulfilled, then  the following estimates hold.
    \begin{equation}\label{e:step-gamma1-cjnorm}
\begin{split}
&\|\Gamma_1\|_0+\|\partial_t\Gamma_1\|_0+\|\partial_t^2\Gamma_1\|_0\leq C\delta^{1/2},\\
&\|\Gamma_1\|_1+\|\partial_t\Gamma_1\|_1\leq C\delta^{1/2}\lambda^{\tau},\\
&\|\partial_s\Gamma_1\|_j\leq C\lambda^{\tau j}, \, j=0,1,\\
&\|\Gamma_1\|_2\leq C\delta^{1/2}\lambda^{2\tau},
\end{split}
\end{equation}
and
\begin{equation}\label{e:step-gamma2-cjnorm}
\begin{split}
&\|\Gamma_2\|_0+\|\partial_t\Gamma_2\|+\|\partial_t^2\Gamma_2\|_0\leq C\delta,\\
&\|\Gamma_2\|_1+\|\partial_t\Gamma_2\|_1\leq C\delta\lambda^\tau,\\
&\|\partial_s\Gamma_2\|_j\leq C\delta^{1/2}\lambda^{\tau j},\, j=0,1,\\
&\|\Gamma_2\|_2\leq C\delta\lambda^{2\tau}.
\end{split}
\end{equation}
\end{lemma}
\noindent The proof of Lemma \ref{Gamma-estimate} is the same as that in \cite{Cao}, but for the completeness, we put it in Appendix. 

So from \eqref{e:vj}, \eqref{eq-phi-cj}, \eqref{e:wj}, and Lemma \ref{Gamma-estimate},  we get
\begin{equation}\label{e:v1-estimate}
\|v^\lozenge-v\|_j\leq C\delta^{1/2}\lambda^{\tau j-\tau}, \,\, j=0, 1, 2,
\end{equation}
and
\begin{equation}\label{e:w1-estimate}
\|w^\lozenge-w\|_j\leq C(\delta^{1/2}\lambda^{\tau j-\tau}+\delta\lambda^{\tau j-\tau})\leq C\delta^{1/2}\lambda^{\tau j-\tau}, j=0, 1, 2.
\end{equation}
Hence by \eqref{eq-vw-c1}, \eqref{eq-vw-c2}, \eqref{e:v1-estimate} and \eqref{e:w1-estimate}, one arrives at
\begin{equation}\label{eq-vw-natural-c2}
\begin{split}
&\|v^\lozenge\|_1+\|w^\lozenge\|_1\leq C,\\
&\|v^\lozenge\|_2+\|w^\lozenge\|_2\leq C\delta^{1/2}\lambda^\tau.
\end{split}
\end{equation}

Next we shall control the $C^j$($j=0, 1$) norms of the error matrix $\mathcal{E}.$ Taking derivatives gives
\begin{align*}
\nabla v^\lozenge&=\nabla v+\lambda^{-\tau}\partial_s\Gamma_1\nabla a+\partial_t\Gamma_1\nabla\Phi_1,\\
\nabla w^\lozenge&=\nabla w-\lambda^{-\tau}\partial_s\Gamma_1\nabla\tilde{v}\otimes\nabla a-\partial_t\Gamma_1\nabla\tilde{v}\otimes\nabla\Phi_1-\lambda^{-\tau}\Gamma_1\nabla^2\tilde{v}\\
&\quad +\lambda^{-\tau}\partial_s\Gamma_2\nabla\Phi_1\otimes\nabla a+\partial_t\Gamma_2\nabla\Phi_1\otimes\nabla\Phi_1.
\end{align*}
Hence, with \eqref{e-gamma12}, we have
\begin{align*}
\mathcal{E}_1=&D(v^\lozenge, w^\lozenge)-\left[D(v, w)+a^2\nabla\Phi_1\otimes\nabla\Phi_1\right]
=\sum_{k=1}^5\mathcal{E}_{1k}
\end{align*}
with
\begin{align*}
\mathcal{E}_{11}:=&\lambda^{-\tau}\left[\partial_t\Gamma_1\partial_s\Gamma_1\sym(\nabla\Phi_1\otimes\nabla a)-\Gamma_1\nabla^2\tilde{v}+\partial_s\Gamma_2\textrm{sym}(\nabla\Phi_1\otimes\nabla a)\right],\\
\mathcal{E}_{12}:=&{\lambda}^{-\tau}\partial_s\Gamma_1\textrm{sym}\left[(\nabla v-\nabla\tilde{v})\otimes\nabla a\right]+\partial_t\Gamma_1\textrm{sym}\left[(\nabla v-\nabla\tilde{v})\otimes\nabla\Phi_1\right],\\
\mathcal{E}_{13}:=&\tfrac12\lambda^{-2\tau}(\partial_s\Gamma_1)^2\nabla a\otimes\nabla a,\\
\mathcal{E}_{14}:=&(V^\lozenge-V)\cdot\mathrm{Id}, \,\,(\text{ here } \Delta V^\lozenge =F(v^\lozenge)) \\
\mathcal{E}_{15}:=&\left[(v^\lozenge-v)\cot\Theta\right]\cdot\mathrm{Id}
\end{align*}
Estimates of the first three errors and their proofs are the same as those in \cite{Cao}.
\begin{lemma}\label{Error-estimate}
Assume all the assumptions of Proposition \ref{p:stage} are fulfilled, then  for $j=0, 1$, the following estimates hold:
 \begin{align*}
&\|\mathcal{E}_{11}\|_j\leq C\delta\lambda^{1-\tau(1-j) },\\
& \|\mathcal{E}_{12}\|_j\leq C\delta\lambda^{1-\tau(1-j)},\\
&\|\mathcal{E}_{13}\|_j\leq C\delta\lambda^{2-j-2\tau(1-j)}.
 \end{align*}
\end{lemma}
\noindent For completeness, we put the proof of Lemma \ref{Error-estimate} in Appendix.

To estimate $\E_{14}$, we set $$\zeta=2(v^\lozenge-v)\nabla \cot\Theta,\quad f=-(v^\lozenge-v)\Delta\cot\Theta.$$ 
Note that
\begin{equation*}
\left\{
\begin{aligned}
\Delta(V^\lozenge-V)=F(v^\lozenge)-F(v)&=\Div \zeta+f, \mbox{ in } \Omega,\\
V^\lozenge-V&=0, \mbox{ on } \partial\Omega.
\end{aligned}
\right.
\end{equation*}
Hence we require the following elliptic estimates. 
\begin{lemma}\label{elliptic}
Let $\Omega\subset\R^n$ be a bounded domain. Suppose $\zeta=(\zeta_1,\cdots,\zeta_n)$ where $\zeta_i\in C^1(\Omega)\cap C^0(\overline{\Omega})$, $i=1,\cdots,n$, and $f\in C^0(\overline{\Omega})$. Let $u\in C^2(\Omega)\cap C^0(\overline{\Omega})$ be the solution to 
\begin{align*}
\begin{cases}
    \Delta u=\Div \zeta+f, &\mbox{ in } \Omega,\\
    \quad u=0, &\mbox{ on } \partial\Omega.
\end{cases}
\end{align*}
Then,
\begin{equation}\label{elliptic-est}
    \|u\|_0\leq C(n, \Omega)(\|f\|_0+\|\zeta\|_0),
\end{equation}
where $C(n, \Omega)$ is a constant depending only on $n$ and $\Omega$.
\end{lemma}
\begin{proof}
Let $u_1$ be the solution of 
\begin{align*}
\begin{cases}
    \Delta u_1=\Div \zeta, &\mbox{ in } \Omega,\\
    \quad u_1=0, &\mbox{ on } \partial\Omega,
\end{cases}
\end{align*}
and $u_2$ be the solution of 
\begin{align*}
\begin{cases}
    \Delta u_2=f, &\mbox{ in } \Omega,\\
    \quad u_2=0, &\mbox{ on } \partial\Omega,
\end{cases}
\end{align*}
then $u=u_1+u_2$, and 
\begin{equation}\label{est-1}
    \|u\|_0\leq\|u_1\|_0+\|u_2\|_0.
\end{equation}
By the maximum principle, 
\begin{equation}\label{est-2}
    \|u_2\|_0\leq C(n, \diam(\Omega))\|f\|_0,
\end{equation}
where $\diam(\Omega)$ is the diameter of $\Omega$. 
Now \eqref{elliptic-est} will follow from \eqref{est-1}, \eqref{est-2} and the claim below:

\noindent\textbf{Claim}: $\|u_1\|_0\leq C(n, \Omega)\|\zeta\|_0.$

Let $\phi$ be the solution of 
\begin{align*}
\begin{cases}
    \Delta \phi=u_1, &\mbox{ in } \Omega,\\
    \quad \phi=0, &\mbox{ on } \partial\Omega,
\end{cases}
\end{align*}
then 
\begin{align*}
    \int_\Omega u_1^2d x&=\int_\Omega u_1\Delta\phi d x=\int_\Omega \phi\Delta u_1 d x=\int_\Omega \phi\Div \zeta d x\\
    &=-\int_\Omega \nabla \phi \cdot \zeta d x\leq\|\nabla\phi\|_{L^2}\|\zeta\|_{L^2}.
\end{align*}
Besides, the maximum principle tells us that
\begin{equation}\label{phi-est}
    \|\phi\|_0\leq C(n, \diam(\Omega))\|u_1\|_0.
\end{equation}
Noting
\begin{align*}
    \int_\Omega |\nabla\phi|^2 d x=-\int_\Omega \phi\Delta\phi d x=-\int_\Omega \phi u_1 d x\leq \|\phi\|_{L^2}\|u_1\|_{L^2},
\end{align*}
we have $$\|\nabla\phi\|_{L^2}\leq \|\phi\|_{L^2}^{\frac{1}{2}}\|u_1\|_{L^2}^{\frac{1}{2}},$$ and further,
\begin{equation*}
    \|u_1\|^2_{L^2}\leq \|\phi\|_{L^2}^{\frac{1}{2}}\|u_1\|_{L^2}^{\frac{1}{2}}\|\zeta\|_{L^2}.
\end{equation*}
Therefore, we get
\begin{equation}\label{u1-est-1}
    \|u_1\|^{\frac{3}{2}}_{L^2}\leq \|\phi\|_{L^2}^{\frac{1}{2}}\|\zeta\|_{L^2}\leq |\Omega|^{\frac{1}{4}}\|\phi\|_0^{\frac{1}{2}}|\Omega|^{\frac{1}{2}}\|\zeta\|_0\leq C(n, \Omega)\|u_1\|^{\frac{1}{2}}_0\|\zeta\|_0.
\end{equation}
Here we apply \eqref{phi-est} in the last inequality.
By Theorem 8.15 in \cite{GT}, we know
$$\|u_1\|_0\leq C(n, \Omega)(\|\zeta\|_0+\|u_1\|_{L^2}),$$
which combined with \eqref{u1-est-1} will imply
\begin{equation}\label{u1-est-2}
\|u_1\|_0\leq C(n, \Omega)(\|\zeta\|_0+\|u_1\|_0^{\frac{1}{3}}\|\zeta\|_0^{\frac{2}{3}}).
\end{equation}
Let $D=\|u_1\|_0$ and $B=\|\zeta\|_0$. We consider the following two cases.

{\it Case 1}: $D^{\frac{1}{3}}B^{\frac{2}{3}}\leq B$. In this case, we have $D\leq B$ and the claim follows.

{\it Case 2}: $D^{\frac{1}{3}}B^{\frac{2}{3}}\geq B$. In this case, by \eqref{u1-est-2} we have
$$D\leq 2C(n, \Omega)D^{\frac{1}{3}}B^{\frac{2}{3}},$$
which also implies $$D\leq C(n, \Omega)B.$$
Then we complete the proof of the claim.
\end{proof}

With Lemma \ref{elliptic} and \eqref{e:v1-estimate},
\begin{align*}
    \|V^\lozenge-V\|_{0}\leq C(\|f\|_0+\|\zeta\|_0)\leq C\|v^\lozenge-v\|_{0}\leq C\delta^{\frac{1}{2}}\lambda^{-\tau},
\end{align*}
and for any $\sigma\in (0, 1)$, by the $C^{1,\sigma}$ estimates ((4.45)-(4.46) in \cite{GT}) , 
\begin{align*}
\|V^\lozenge-V\|_{1+\sigma}\leq C(\|\zeta\|_\sigma+\|f\|_0)\leq& C\|v-v^\lozenge\|_\sigma\leq C\|v-v^\lozenge\|_0^{1-\sigma}\|v-v^\lozenge\|_1^\sigma\\
\leq&C(\delta^{\frac{1}{2}})^\sigma(\delta^{\frac{1}{2}}\lambda^{-\tau})^{1-\sigma}=C\delta^{\frac{1}{2}}\lambda^{-\tau(1-\sigma)}.
\end{align*}
Then,
\begin{align*}
\|\mathcal{E}_{14}\|_0&=\|V-V^\lozenge\|_0\leq C\delta^{\frac{1}{2}}\lambda^{-\tau}\leq C\delta\lambda^{1-\tau},\\
    \|\mathcal{E}_{14}\|_1&=\|V-V^\lozenge\|_1\leq\|V-V^\lozenge\|_{1+\sigma}\leq C\delta^{\frac{1}{2}}\lambda^{-\tau(1-\sigma)}\leq C\delta\lambda.
\end{align*}

 It remains to estimates $\mathcal{E}_{15}$. From \eqref{e:v1-estimate}, for $j=0, 1$, we can derive 
\begin{align*}
    \|\mathcal{E}_{15}\|_j\leq C\|v-v^\lozenge\|_j\leq C\delta^{\frac{1}{2}}\lambda^{\tau j-\tau}\leq\delta\lambda^{1-\tau(1-j)},
\end{align*}
where we have used $\delta^{1/2}\lambda>1.$
Putting the estimates of $\mathcal{E}_{1k}, k=1,\cdots, 5$ together, we obtain
 \begin{equation}\label{eq-error1}
 \|\mathcal{E}_1\|_0\leq C\delta\lambda^{1-\tau},\quad \|\mathcal{E}_1\|_1\leq C\delta\lambda.
\end{equation}

\smallskip

{\it Step 3.Adding the second rank-one matrix.} We also mollify $(v^\lozenge, w^\lozenge)$ to get $(\tilde v^\lozenge, \tilde w^\lozenge)$ at the length-scale $\lambda^{1-2\tau}$. Then by \eqref{eq-vw-natural-c2} and Lemma \ref{holder}, one also has
\begin{equation}\label{eq-vw-natural-tilde}
\begin{split}
&\|\tilde{v}^\lozenge-v^\lozenge\|_1+\|\tilde{w}^\lozenge-w^\lozenge\|_1\leq C\delta^{1/2}\lambda^{1-\tau},\\
&\|\tilde{v}^\lozenge\|_{2+j}+\|\tilde{w}^\lozenge\|_{2+j}\leq C\delta^{1/2}\lambda^{\tau+(2\tau-1)j},
 \end{split}
\end{equation}
for any $j\geq0.$ Set
\begin{align*}
v^\star(x)=&v^\lozenge(x)+\lambda^{1-2\tau}\Gamma_1(a(x), \lambda^{2\tau-1} \Phi_2),\\
w^\star(x)=&w^\lozenge(x)-\lambda^{1-2\tau}\Gamma_1(a(x), \lambda^{2\tau-1}\Phi_2)\nabla\tilde{v}^\lozenge(x)\\
&+\lambda^{1-2\tau}\Gamma_2(a(x),\lambda^{2\tau-1}\Phi_2)\nabla\Phi_2.
\end{align*}
By the definitions of $\Gamma_i, i=1, 2$, we get
$$(v^\star, w^\star)=(v^\lozenge, w^\lozenge) \textit{ on } \overline\Omega\setminus\textit{supp }a.$$
Then \eqref{eq-support} follows from \eqref{eq-vw-natural-1} and
$$\textit{dist}(\textrm{supp }a, \textrm{supp }\rho)=\lambda^{-\tau}.$$
Furthermore, similar to Step 2,  using \eqref{eq-a-c2},\eqref{eq-phi-cj} and \eqref{eq-vw-natural-c2}, one is able to derive
\begin{equation}\label{e:v-flat-j}
\begin{split}
\|v^\star-v^\lozenge\|_0\leq &\lambda^{1-2\tau}\|\Gamma_1(a(\cdot),\lambda^{2\tau-1}\Phi_2(\cdot))\|_0\leq C\delta\lambda^{1-2\tau},\\
\|v^\star-v^\lozenge\|_1\leq &\lambda^{1-2\tau}\|\Gamma_1(a(\cdot),\lambda^{2\tau-1}\Phi_2(\cdot))\|_1\leq C\delta,\\
\|v^\star-v^\lozenge\|_2\leq &\lambda^{1-2\tau}\|\Gamma_1(a(\cdot),\lambda^{2\tau-1}\Phi_2(\cdot))\|_2\leq C\delta\lambda^{2\tau-1}.
\end{split}
\end{equation}
and
\begin{equation}\label{e:w-flat-j}
\begin{split}
\|w^\star-w^\lozenge\|_j\leq &C(j)\lambda^{1-2\tau}(\|\Gamma_1(a(\cdot),\lambda^{2\tau-1}\Phi_2(\cdot))\|_j\|\tilde v^\lozenge\|_1\\
&\quad +\|\Gamma_1(a(\cdot),\lambda^{2\tau-1}\Phi_2(\cdot))\|_0\|\tilde v^\lozenge\|_{j+1}\\
&\quad +\|\Gamma_2(a(\cdot),\lambda^{2\tau-1}\Phi_2(\cdot))\|_j\|\nabla\Phi_2\|_0\\
&\quad +\|\Gamma_2(a(\cdot),\lambda^{2\tau-1}\Phi_2(\cdot))\|_0\|\nabla\Phi_2\|_j)\\
\leq &C(j)\delta^{1/2}\lambda^{(2\tau-1)(j-1)},\,\, j=0,1,2.
\end{split}
\end{equation}
By \eqref{e:v1-estimate}, \eqref{e:w1-estimate}, \eqref{e:v-flat-j} and \eqref{e:w-flat-j}, we have
\begin{align*}
\|v^\star-v\|_0+\|w^\star-w\|_0\leq& \|v^\star-v^\lozenge\|_0+\|v^\lozenge-v\|_0+\|w^\star-w^\lozenge\|_0+\|w^\lozenge-w\|_0\\
\leq &C(\delta^{1/2}\lambda^{-\tau}+\delta^{1/2}\lambda^{1-2\tau})\leq C\delta^{1/2}\lambda^{-\tau},\\
\|v^\star-v\|_1+\|w^\star-w\|_1\leq&C\delta^{1/2},
\end{align*}
which imply \eqref{eq-vw-diff-0} and \eqref{eq-vw-diff-1} respectively. Besides, from
\begin{align*}
\|v^\star-v\|_2+\|w^\star-w\|_2\leq& \|v^\star-v^\lozenge\|_2+\|v^\lozenge-v\|_2+\|w^\star-w^\lozenge\|_2+\|w^\lozenge-w\|_2\\
\leq& C(\delta^{1/2}\lambda^\tau+\lambda^{2\tau-1})\leq C\delta^{1/2}\lambda^{2\tau-1},
\end{align*}
 and \eqref{eq-vw-c2} we also get \eqref{eq-vw-diff-2} since $\tau>1$. 
 
Finally, we shall control the second deficit. A direct calculation implies
\begin{align*}\mathcal{E}_2:&=D(v^\star, w^\star)
-\left[D(v^\lozenge, w^\lozenge)+a^2\nabla\Phi_2\otimes\nabla\Phi_2\right]=\sum_{k=1}^5\mathcal{E}_{2k}
\end{align*}
with
\begin{align*}
\mathcal{E}_{21}:=&\lambda^{1-2\tau}\left[\partial_t\Gamma_1\partial_s\Gamma_1\textrm{sym}(\nabla\Phi_2\otimes\nabla a)-\Gamma_1\nabla^2\tilde{v}^\lozenge+\partial_s\Gamma_2\textrm{sym}(\nabla\Phi_1\otimes\nabla a)\right],\\
\mathcal{E}_{22}:=&{\lambda}^{-\tau}\partial_s\Gamma_1\textrm{sym}\left[(\nabla v^\lozenge-\nabla\tilde{v}^\lozenge)\otimes\nabla a\right]+\partial_t\Gamma_1\textrm{sym}\left[(\nabla v^\lozenge-\nabla\tilde{v}^\lozenge)\otimes\nabla\Phi_1\right],\\
\mathcal{E}_{23}:=&\tfrac12\lambda^{-2\tau}(\partial_s\Gamma_1)^2\nabla a\otimes\nabla a,\\
\mathcal{E}_{24}:=&(V^\star-V^\lozenge)\cdot\mathrm{Id},\,\, (\text{here }\Delta V^\star=F(v^\star))\\
\mathcal{E}_{25}:=&\left[(v^\star-v^\lozenge)\cot\Theta\right]\cdot\mathrm{Id}.
\end{align*}
Repeating the estimating procedure in Step 2, we are able to get
\begin{equation}\label{eq-error2}
\|\mathcal{E}_2\|_0\leq C\delta\lambda^{\tau-(2\tau-1)}=C\delta\lambda^{1-\tau},\quad
\|\mathcal{E}_2\|_1\leq C\delta\lambda^{\tau}.
\end{equation}
In fact, when  bounding the $C^1$ norm of the second matrix error, we only need to keep in mind that $\lambda^{2\tau-1}$(and $\lambda^\tau$) plays the same role as $\lambda^\tau$(and $\lambda$ repsectively) in Step 2. Therefore,  the final matrix error can be calculated as follows:
\begin{align*}
\mathcal{E}&=D(v^\star,w^\star)-(D(v,w)+\rho^2(\textrm{Id}+H))\\
&=D(v^\star,w^\star)-(D(v^\lozenge,w^\lozenge)+a^2\nabla\Phi_2\otimes\nabla\Phi_2)\\
&\qquad +D(v^\lozenge,w^\lozenge)-(D(v,w)+a^2\nabla\Phi_1\otimes\nabla\Phi_1)+\tilde{h}-h\\
&=\mathcal{E}_1+\mathcal{E}_2+\tilde{h}-h.
\end{align*}
After collecting the estimates of $\tilde{h}-h$ in Step 1, \eqref{eq-error1} and \eqref{eq-error2}, we can achieve \eqref{eq-deficit-c01} and complete the proof.

\end{proof}

\section{Proof of Theorem \ref{infity-sol} (2)}\label{pfmain}
As discussed at the beginning of Section \ref{s-stage},  we show Theorem \ref{infity-sol} (2) by seeking infinitely many symmetric $2\times2$ matrixes $A$'s and $C^{1,\alpha}(\alpha<\frac15)$ functions $v$'s such that
\[\begin{cases} -\mbox{curl curl }A=1,\\ A=D(v, w)=\sym \nabla w+\frac{1}{2}\nabla v\otimes \nabla v-(v\cot\Theta)\cdot\mathrm{Id}+V\cdot\mathrm{Id}, \end{cases}\]
where $V$ satisfies the following Dirichlet problem:
 \begin{equation}\label{bigV}
\left\{
\begin{aligned}
  \Delta V&=F(v)=2\nabla v\cdot\nabla \cot\Theta+v\Delta\cot\Theta, \mbox{ in } \Omega,\\ 
  V&=0, \mbox{ on } \partial\Omega,
\end{aligned}
\right.
\end{equation}
We will solve the first system about $A$ together with a subsolution by seeking solutions to some Poisson equations and the second system by constructing a sequence of adapted subsolutions via convex integration.

\begin{proof}[Proof of Theorem \ref{infity-sol} (2)]
We separate the proof into four steps.

{\it Step 1. Matrix and initial subsolutions.}  We construct the matrix $A$ and initial subsolutions at the same time. 
Let $u$ be the solution to the following Dirichlet problem,
\begin{equation*}
\left\{
\begin{aligned}
        \Delta u &=0, \mbox{ in } \Omega,\\
        u&=g, \mbox{ on } \partial\Omega,
\end{aligned}
    \right.
\end{equation*}
and $\psi$ be the solution of the following Dirichlet problem,
\begin{equation*}
\left\{
\begin{aligned}
        -\Delta \psi &=1-\det\nabla^2 u, \mbox{ in } \Omega,\\
        \psi&=0, \mbox{ on } \partial\Omega.
\end{aligned}
    \right.
\end{equation*}
Define 
$$A=\psi\mathrm{Id}+\frac{1}{2}\nabla u\otimes\nabla u+(-u\cot\Theta+U)\mathrm{Id},$$
where $U$ is the solution to 
\begin{equation*}
\left\{
\begin{aligned}
        \Delta U &=F(u):=2\nabla u\cdot\nabla \cot\Theta+u\Delta\cot\Theta, \mbox{ in } \Omega,\\
        U&=0, \mbox{ on } \partial\Omega.
\end{aligned}
    \right.
\end{equation*}

Note that $|\sin\Theta|\geq c_2>0$ as in the assumptions of Theorem \ref{infity-sol}. So $\cot\Theta$ is a $C^{2,\kappa}$ bounded function. 
The existence and regularity of $u\in C^{2,\kappa}$, $\psi\in C^{2,\kappa}$ and $U\in C^{2,\kappa}$ are guaranteed by the elliptic theory. 

Obviously, $A\in C^{1, \kappa}(\overline{\Omega})$. Besides,
\begin{align*}
    -\curl\, \curl\, A&=-\Delta \psi +\det\nabla^2 u-\Delta(-u\cot\Theta+U)\\
    &=1+\Delta(u\cot\Theta)-F(u)\\&=1+\cot\Theta\Delta u=1.
\end{align*}
With $\Delta u=0$, we know $1-\det\nabla^2u=1+(\partial^2_{11}u)^2+(\partial^2_{12}u)^2>0$, and hence by the strong maximum principle, we have $$\psi(x)>0, \mbox{ for any } x\in \Omega.$$
Set the initial subsolution 
\[(v_0, w_0)=(u, 0).\]
Then, $$A-D(v_0,w_0)=A-D(u, 0)=\psi\mathrm{Id}:=\rho_0^2(\mathrm{Id}+H_0),$$
where $\rho_0=\sqrt{\psi}$ and $H_0$ is the zero matrix. Moreover, $v_0|_{\partial\Omega}=u|_{\partial\Omega}=g$.  

Note that $(v_0, w_0)$ and $A$ are not unique. For example, taking $v_0$ as before and $w_0\in\R^2$ as an arbitrary constant vector will also work.

\smallskip

{\it Step 2. Iteration. } 
We define the amplitude sequence $\delta_q$ and the frequency sequence $\lambda_q$ inductively by
$$\delta_1=\max_{x\in\overline\Omega}\rho_0^2,\quad \lambda_{q+1}=\lambda_q^b,\quad \lambda_q=M\delta_q^{-\frac1{2\beta}},$$
where $\beta$ is a arbitrary fixed constant in $(0,\frac15)$, and  
\begin{equation}\label{eq-b}
b=1+\frac{9\sigma}{1-5\beta} \quad (\text{then }5\beta(b-1)+9\sigma=b-1), \end{equation}
and $M$ is a large constant to be fixed. We then decomposes $\overline\Omega$ according to the level sets of $\rho_0$,
$$
\Omega_q=\left\{x\in\overline\Omega:\,\rho_0(x)\geq2\delta_{q+2}^{1/2}\right\},\,\,
\tilde\Omega_q=\left\{x\in\overline\Omega:\,\rho_0(x)\geq\frac32\delta_{q+2}^{1/2}\right\}
$$
for any $q\in\mathbb{N}$. 
With Proposition \ref{p:stage}, for any $q\in\N$, we inductively construct $(v_q, w_q, \rho_q, H_q)$ such that the following hold:
\begin{itemize}
\item[$(1)_q$] For any $x\in\overline\Omega$, it holds that
$D_q=:A-D(v_q, w_q)=\rho_q^2(\textrm{Id}+H_q).$
\item[$(2)_q$] If $x\in\overline\Omega\setminus\tilde\Omega_q$, then $(v_q, w_q, \rho_q, H_q)(x)=(v_0, w_0, \rho_0, H_0)(x).$
\item[$(3)_q$] The following estimates hold in $\overline\Omega$:

\begin{align*}
\|\rho_q\|_0\leq2\delta_{q+1}^{1/2}, \quad \|H_q\|_0\leq4\lambda_{q+1}^{-\frac\alpha b}.
\end{align*}
\item[$(4)_q$] If $x\in\tilde\Omega_j$ for some $j\geq q,$ we get the following estimates:
\begin{align*}
&\rho_q\geq\frac32\delta_{j+2}^{1/2},\quad|\nabla\rho_q|\leq \delta_{j+1}^{1/2}\lambda_{j+2},\quad
\frac{|\nabla\rho_q|}{\rho_q}\leq\lambda_{j+2},\\
&|\nabla^2v_q|+|\nabla^2 w_q|\leq \delta_{j+1}^{1/2}\lambda_{j+2}, \quad
|\nabla H_q|\leq \lambda_{j+2}^{1-\frac\alpha b}.
\end{align*}
\item[$(5)_q$]  For any $x\in\overline\Omega,$ $q\geq1,$
\begin{equation}\label{e:vqdifference}
\|v_q-v_{q-1}\|_j+\|w_q-w_{q-1}\|_j\leq \overline C\delta_q^{1/2}\lambda_q^{j-1}, j=0, 1,
\end{equation}
where $\overline C$ is the constant in the conclusion of Proposition \ref{p:stage}.
\end{itemize}

Indeed, $(v_0, w_0, \rho_0, H_0)$ has been constructed in Step 1 and it is not hard to verify $(1)_0-(4)_0$. Similar to Step 2 of the proof of Proposition 2.1 in \cite{Cao}, by using Proposition \ref{p:stage} we are able to inductively get $(v_q, w_q, \rho_q, H_q)$ for $q\geq1$. Since the proof is exactly the same as in \cite{Cao}, we omit the details.

\smallskip

{\it Step 3. Taking limits.} 
By $(3)_q$ and $(5)_q$, we know $(\rho_q, H_q)$ is a Cauchy sequence in $C^0(\overline{\Omega})$, and $(v_q, w_q)$ is a Cauchy sequence in $C^1(\overline\Omega)$, respectively. Hence we can assume
$$(v_q, w_q, \rho_q, H_q)\rightarrow(\bar{v},\bar w, 0, 0)  \text{ uniformly in } \overline\Omega,$$ 
for some functions $\bar v$ and $\bar w$. 

Moreover, for any $\alpha\in (0, \frac{1}{5})$ and $\varepsilon>0$, we can choose $\beta\in (\alpha, \frac{1}{5})$ in Step 2. By $(4)_q-(5)_q$, $(4)_{q+1}$ and interpolation, we have for any $\alpha\in (0, \frac{\beta}{b^2})$,
$$\|v_q-v_{q+1}\|_{1+\alpha}+\|w_q-w_{q+1}\|_{1+\alpha}\leq 4\delta_{q+1}^{(1-\alpha)/2}\delta_{q+2}^{\alpha/2}\lambda_{q+3}^{\alpha}\leq M^\beta\lambda_{q+1}^{\alpha b^2-\beta},$$
which implies $(v_q, w_q)$ uniformly converges to $(\bar v, \bar w)$ in $C^{1, \alpha}(\overline{\Omega})$. Since $\alpha$ and $b$ can be taken arbitrarily close to $0$ and $1$ by \eqref{eq-b}, respectively, $(\bar{v}, \bar w)\in C^{1, \alpha}({\Omega}),$ for any $\alpha<\beta$. From $(5)_q$, for any $\epsilon>0,$ we also have
\begin{align*}
\|\bar{v}-v_0\|_0+\|\bar w-w_0\|_0\leq&\sum_{q=0}^{\infty}\|v_q-v_{q+1}\|_0+\|w_{q+1}-w_q\|_0\\
\leq&\sum_{q=0}^\infty \overline C\delta_{q+1}^{1/2}\lambda_{q+1}^{-1}
\leq \overline CM^{-1}\leq\epsilon,
\end{align*}
if $M$ is taken larger.

{\it Step 4. Very weak solutions.} We will show $\bar v$ is a very weak solution to \eqref{eq-dirich-lag}. On the one hand, by construction, it is easy to see $(\bar v, \bar w)|_{\partial\Omega}=(v_0, w_0)|_{\partial\Omega}$. On the other hand, $\rho_q\to0$ implies $A-D(\bar v, \bar w)=0.$ Thus for any $\phi\in C^\infty_0(\Omega)$, by integration by parts and $\mbox{curl curl }A=-1$, we have
\begin{align*}
  -\int_\Omega \phi\sin\Theta d x=&\int_\Omega \mathrm{Trace} (H_{\phi\sin\Theta}\cdot A) d x \\=&\int_\Omega \mathrm{Trace} (H_{\phi\sin\Theta}\cdot \frac{1}{2}\nabla \bar v\otimes \nabla \bar v) d x-\int_\Omega \bar v\Delta(\phi\cos\Theta) d x,
\end{align*}
where we have used $\mbox{curl curl}(\sym\nabla \bar w)=0$ and the notation 
\[
  H_f=\left (\begin{array}{rr}
    \partial_{22}^2f & -\partial_{12}^2f   \\
    -\partial_{12}^2f & \partial_{11}^2f   \\
    \end{array}\right) \text{ for any $C^2$ function }f.
\]
Then it implies
\begin{align*}
0=\int_\Omega\partial_1 \bar v\partial_2\bar v\partial_{12}^2(\phi\sin\Theta)&-\frac{1}{2}(\partial_1\bar v)^2\partial_{22}^2(\phi\sin\Theta)-\frac{1}{2}(\partial_2\bar v)^2\partial_{11}^2(\phi\sin\Theta) d x\\&+\int_\Omega \bar v\Delta (\phi \cos\Theta)d x-\int_\Omega \phi\sin\Theta d x.
\end{align*}
Finally, since $\epsilon$ can be arbitrary and the initial subsolution is also not unique (note each $(v_q, w_q)$ can be an initial subsolution), \eqref{eq-dirich-lag} admits infinitely many very weak solutions $\bar v$'s. We then complete the proof.
\end{proof}

\appendix
\section{Useful lemmas and proofs of Lemmas \ref{Gamma-estimate}-\ref{Error-estimate}}

In this appendix, we provide the sketch proofs of Lemmas \ref{Gamma-estimate}-\ref{Error-estimate} for completeness, which are as same as those in \cite{Cao}. We begin with the following two useful lemmas.
\begin{lemma}[\cite{CDS12}]\label{holder} 
For any $r, s\geq 0$ and $0<\alpha\leq 1$, we have
\begin{align*}
    &[h*\phi_l]_{r+s}\leq Cl^{-s}[h]_r;\\
    &\|h-h*\phi_l\|_r\leq Cl^{1-r}[h]_1, \mbox{ for } 0\leq r\leq 1;\\
    &\|h-h*\phi_l\|_j\leq Cl^{2-j}\|h\|_2, \mbox{ for } j=0, 1;\\
    &\|(h_1h_2)*\phi_l-(h_1*\phi_l)(h_2*\phi_l)\|_r\leq Cl^{2\alpha-r}\|h_1\|_\alpha\|h_2\|_\alpha,
\end{align*}
with the constant $C$ depending only on $s, r, \alpha, \phi$.
\end{lemma}
\begin{lemma}[\cite{Cao}]\label{diag}
  Let $\Sigma\subset\R^2$ be a simply connected open bounded set with a smooth boundary and $H:\Sigma\to \R^{2\times 2}$ a smooth $2\times2$ symmetric positive definite matrix-valued function, such that for some $0<\alpha<1$, and $c, M\geq1$,
  $$c^{-1}\mathrm{Id}\leq H\leq c\mathrm{Id}, \,\, \|H\|_{C^\alpha(\Sigma)}\leq M.$$
 Then there exist a smooth diffeomorphism $\Phi:\Sigma\to \R^2$ and a smooth positive function $\vartheta:\Sigma\to \R$ satisfying $$H=\vartheta^2(\nabla \Phi_1\otimes\nabla\Phi_1+\nabla \Phi_2\otimes\nabla\Phi_2).$$
 Moreover the following estimates hold:
\begin{align*}
  &\det(\nabla \Phi(x))\geq c_0,\,\, \vartheta(x)\geq c_0, \mbox{ for all } x\in\Sigma, \\&\|\vartheta\|_{j+\alpha}+\|\nabla\Phi\|_{j+\alpha}\leq C(j)\|H\|_{j+\alpha},\,\, j\in \N,
\end{align*}
where the constants $c_0>0$, $C(j)\geq 1$ depend only on $\alpha, c, M$ and on $\Sigma$.
\end{lemma}

\begin{proof}[\bf A sketch proof of Lemma \ref{Gamma-estimate}] 
Recall the definition of $\Gamma_i: \R\times\R\to\R$, $i=1, 2$:
\[
\Gamma_1(s, t)=\frac{s}{\pi}\sin(2\pi t),\quad \Gamma_2(s, t)=-\frac{s^2}{4\pi}\sin(4\pi t).\]
For any nonnegative integer $k$, it is not hard to show the following estimates: 
\begin{equation}\label{eq-gamma-st}
  \begin{split}
    |\partial_t^k\Gamma_1(s, t)|+|\partial_s\partial^k_t\Gamma_2(s, t)|&\leq C|s|,\\
    |\partial_s\partial_t^k\Gamma_1(s, t)|&\leq C,\\
    |\partial_t^k\Gamma_2(s, t)|&\leq Cs^2.
  \end{split}
\end{equation}
The derivatives of $\Gamma_i$, $i=1, 2$ could be calculated directly by the chain rule, then the desired estimates directly follow from the assumptions on $\rho, H$, Lemma \ref{diag} and \eqref{eq-gamma-st}. 
\end{proof}

\begin{proof}[\bf A sketch proof of Lemma \ref{Error-estimate}]
By \eqref{eq-phi-cj}, \eqref{eq-a-c2}, \eqref{e:tilde-vw-C1}, \eqref{e:tilde-vw-C2}, \eqref{e:step-gamma1-cjnorm} and \eqref{e:step-gamma2-cjnorm}, we get
\begin{align*}
\|\mathcal{E}_{11}\|_0\leq&\lambda^{-\tau}(\|\partial_t\Gamma_1\|_0\|\partial_s\Gamma_1\|_0\|\nabla a\|_0\|\nabla\Phi_1\|_0+\|\Gamma_1\|_0\|\nabla^2\tilde{v}\|_0\\
&\quad+\|\partial_s\Gamma_2\|_0\|\nabla a\|_0\|\nabla\Phi_1\|_0)
\leq C\delta\lambda^{1-\tau},
\end{align*}
and
\begin{align*}
 \|\mathcal{E}_{12}\|_0\leq&\lambda^{-\tau}\|\partial_s\Gamma_1\|_0\|\nabla v-\nabla\tilde{v}\|_0\|\nabla a\|_0+\|\partial_t\Gamma_1\|_0\|\nabla v-\nabla\tilde{v}\|_0\|\nabla\Phi_1\|_0\\
 \leq& C(\delta\lambda^{2-2\tau}+\delta\lambda^{1-\tau})
 \leq C\delta\lambda^{1-\tau},
 \end{align*}
and
\begin{align*}
\|\mathcal{E}_{13}\|_0\leq&\tfrac12\lambda^{-2\tau}\|\partial_s\Gamma_1\|_0^2\|\nabla a\|_0^2\leq C\delta\lambda^{2-2\tau}.
 \end{align*}
 The gradients of $\E_{11}$, $\E_{12}$ and $\E_{13}$ could be calculated directly by the chain rule, then the corresponding estimates can be got similarly.
\end{proof}
\bibliography{bib}
\bibliographystyle{plain}
\end{document}